\documentclass[a4paper,12pt,twoside,leqno,final]{amsart}
\usepackage{amsmath}
\usepackage{amssymb}

\newtheorem{thm}{Theorem}[section]
\newtheorem{lem}[thm]{Lemma}
\newtheorem{defn}[thm]{Definition}

\theoremstyle{definition}
\newtheorem{exmp}{Example}[section]

\newcommand{\C}{{\mathbb C}}
\newcommand{\D}{{\mathbb D}}
\newcommand{\R}{{\mathbb R}}
\newcommand{\T}{{\mathbb T}}
\newcommand{\Z}{{\mathbb Z}}

\newcommand{\La}{\Lambda}

\newcommand{\f}{\frac}
\newcommand{\ov}{\overline}
\newcommand{\wt}{\widetilde}
\newcommand{\al}{\alpha}
\newcommand{\be}{\beta}

\newcommand{\ga}{\gamma}

\newcommand{\de}{\delta}
\newcommand{\la}{\lambda}
\newcommand{\ze}{\zeta}
\renewcommand{\th}{\theta}

\newcommand{\ph}{\varphi}

\newcommand{\const}{\text{\rm const}}

\numberwithin{equation}{section}

\title[Nearly outer functions]
{Nearly outer functions as extreme\\ 
points in punctured Hardy spaces}
\author{Konstantin M. Dyakonov}
\address{Departament de Matem\`atiques i Inform\`atica, IMUB, BGSMath, Universitat de Barcelona, Gran Via 585, E-08007 Barcelona, Spain}
\address{ICREA, Pg. Llu\'is Companys 23, E-08010 Barcelona, Spain}
\email{konstantin.dyakonov@icrea.cat}
\thanks{Supported in part by grant MTM2017-83499-P from El Ministerio de Ciencia e Innovaci\'on (Spain) and grant 2017-SGR-358 from AGAUR (Generalitat de Catalunya).}

\begin{document}
\begin{abstract} The Hardy space $H^1$ consists of the integrable functions $f$ on the unit circle whose Fourier coefficients $\widehat f(k)$ vanish for $k<0$. We are concerned with $H^1$ functions that have some additional (finitely many) holes in the spectrum, so we fix a finite set $\mathcal K$ of positive integers and consider the \lq\lq punctured" Hardy space 
$$H^1_{\mathcal K}:=\{f\in H^1:\,\widehat f(k)=0\,\,\,\text{\rm for all }\,
k\in\mathcal K\}.$$
We then investigate the geometry of the unit ball in $H^1_{\mathcal K}$. In particular, the extreme points of the ball are identified as those unit-norm functions in $H^1_{\mathcal K}$ which are not too far from being outer (in the appropriate sense). This extends a theorem of de Leeuw and Rudin that deals with the classical $H^1$ and characterizes its extreme points as outer functions. We also discuss exposed points of the unit ball in $H^1_{\mathcal K}$.
\end{abstract}

\maketitle

\section{Introduction and results} 

We shall be concerned with certain Hardy-type spaces on the circle
$$\T:=\{\ze\in\C:\,|\ze|=1\}.$$ 
The functions to be dealt with are complex-valued and live almost everywhere on $\T$, which is endowed with normalized arc length measure. The spaces $L^p=L^p(\T)$ with $0<p\le\infty$ are then defined in the usual way. Among these, of special relevance to us is $L^1$, the space of integrable functions on $\T$ with norm 
\begin{equation}\label{eqn:Lonenorm}
\|f\|_1:=\f 1{2\pi}\int_\T|f(\ze)|\,|d\ze|,
\end{equation}
as well as some of its subspaces, to be specified shortly. 

\par For a given function $f\in L^1$, we consider the sequence of its {\it Fourier coefficients}
$$\widehat f(k):=\f 1{2\pi}\int_\T\ov\ze^kf(\ze)\,|d\ze|,\qquad k\in\Z,$$
and the set 
$$\text{\rm spec}\,f:=\{k\in\Z:\,\widehat f(k)\ne0\},$$
known as the {\it spectrum} of $f$. Now, the {\it Hardy space} $H^1$ is defined by $$H^1:=\{f\in L^1:\,\text{\rm spec}\,f\subset\Z_+\},$$
where $\Z_+:=\{0,1,2,\dots\}$, and is equipped with the $L^1$ norm \eqref{eqn:Lonenorm}. Equivalently (see \cite[Chapter II]{G}), $H^1$ consists of all $L^1$ functions whose Poisson integral (i.e., harmonic extension) is holomorphic on the disk 
$$\D:=\{z\in\C:\,|z|<1\}.$$
Using this extension, we therefore may---and will---also treat elements of $H^1$ as holomorphic functions on $\D$. 

\par Our starting point is a beautiful theorem of de Leeuw and Rudin, which describes the extreme points of the unit ball in $H^1$. This will be stated in a moment, whereupon certain finite-dimensional perturbations of that result will be discussed. But first we have to fix a bit of terminology and notation. 

\par Given a Banach space $X=(X,\|\cdot\|)$, we write  
$$\text{\rm ball}(X):=\{x\in X:\,\|x\|\le1\}$$
for the closed unit ball of $X$. A point of $\text{\rm ball}(X)$ is said to be {\it extreme} for the ball if it is not the midpoint of any nondegenerate line segment contained in $\text{\rm ball}(X)$. Of course, every extreme point $x$ of $\text{\rm ball}(X)$ satisfies $\|x\|=1$. 

\par Further, we need to recall the canonical (inner-outer) factorization theorem for $H^1$ functions. By definition, a function $I$ in $H^\infty:=H^1\cap L^\infty$ is {\it inner} if $|I|=1$ a.e. on $\T$. Also, a non-null function $F\in H^1$ is termed {\it outer} if 
$$\log|F(0)|=\f 1{2\pi}\int_\T\log|F(\ze)|\,|d\ze|.$$
It is well known that the general form of a function $f\in H^1$, $f\not\equiv0$, is given by 
\begin{equation}\label{eqn:fif}
f=IF,
\end{equation}
where $I$ is inner and $F$ is outer. Moreover, the two factors are uniquely determined by $f$ up to a multiplicative constant of modulus $1$. We refer to \cite{G} or \cite{Hof} for a systematic treatment of these matters in the framework of general $H^p$ spaces. 

\par Now, the de Leeuw--Rudin theorem states that the extreme points of $\text{\rm ball}(H^1)$ are precisely the outer functions $F\in H^1$ with $\|F\|_1=1$. In addition to the original paper \cite{dLR}, we cite \cite[Chapter IV]{G} and \cite[Chapter 9]{Hof}, where alternative presentations are given. 

\par Our purpose here is to find out what happens for subspaces of $H^1$ that consist of functions with smaller spectra. We do not want to deviate too much from the classical $H^1$, so we consider the case of finitely many additional \lq\lq spectral holes." Precisely speaking, we fix some positive integers 
$$k_1<k_2<\dots<k_M$$
and move from generic functions $f\in H^1$ to those satisfying 
$$\widehat f(k_1)=\dots=\widehat f(k_M)=0.$$
The functions that arise have their spectra contained in the \lq\lq punctured" set $\Z_+\setminus\mathcal K$, where 
\begin{equation}\label{eqn:defmathcalk}
\mathcal K:=\{k_1,\dots,k_M\}.
\end{equation}
The subspace they populate is therefore 
$$H^1_{\mathcal K}:=\{f\in H^1:\,\text{\rm spec}\,f\subset
\Z_+\setminus\mathcal K\}$$
(normed by \eqref{eqn:Lonenorm} again), which might be called the {\it punctured}, or rather {\it $\mathcal K$-punctured}, {\it Hardy space}. The number $M:=\#\mathcal K$ was so far assumed to be a positive integer, but it is convenient to allow the value $M=0$ as well. In the latter case, the convention is that $\mathcal K=\emptyset$, so $H^1_{\mathcal K}=H^1$ and we are back to the classical situation. 

\par In what follows, we are concerned with the geometry of 
$\text{\rm ball}(H^1_{\mathcal K})$, the unit ball in $H^1_{\mathcal K}$, primarily with the structure of its extreme points. Recently, a similar study was carried out in \cite{DAdv2021} for a certain family of {\it finite-dimensional} subspaces in $H^1$; each of those was associated with a finite set $\La\subset\Z_+$ and consisted of the polynomials $p$ with $\text{\rm spec}\,p\subset\La$. By contrast, our current spaces $H^1_{\mathcal K}$ are of {\it finite codimension} in $H^1$, so we are now moving to the opposite extreme. The intermediate cases---not treated here---might also be of interest. 

\par We briefly mention some other types of subspaces in $H^1$ where the geometry of the unit ball has been investigated. Namely, this was done for the so-called model subspaces \cite{DKal, DSib}, and more generally, for kernels of Toeplitz operators \cite{DPAMS, DAMP}. Spaces of polynomials of fixed degree---and their Paley--Wiener type analogues on the real line---fit into that framework and were studied in more detail; see \cite{DMRL2000}. However, spaces of functions with spectral gaps, such as $H^1_{\mathcal K}$ (or the lacunary polynomial spaces from \cite{DAdv2021}), are different in nature and require a new method. In particular, one of the difficulties to be faced in the \lq\lq punctured spectrum" case is that such spaces no longer admit division by inner factors. 

\par Our criterion for a unit-norm function $f\in H^1_{\mathcal K}$ to be an extreme point of $\text{\rm ball}(H^1_{\mathcal K})$ will be stated in terms of the function's canonical factorization \eqref{eqn:fif}. The set $\mathcal K$ of forbidden frequencies being finite, it seems natural to expect that the criterion should be fairly reminiscent of its classical prototype (i.e., the de Leeuw--Rudin theorem), so the functions that obey it are presumably not too far from being outer. We shall indeed identify the extreme points $f$ of $\text{\rm ball}(H^1_{\mathcal K})$ as \lq\lq nearly outer" functions of norm $1$. Specifically, we shall see that the inner factors $I$ of such functions are rather tame (rational, and with a nice bound on the degree); in addition, there is an interplay between the two factors, $I$ and $F$, to be described below. 

\par Now, let us recall that every inner function has the form $BS$, where $B$ is a {\it Blaschke product} and $S$ a {\it singular inner function}. The former factor is determined by its zero sequence in $\D$, while the latter has no zeros and is generated---in a certain canonical way---by a singular measure on $\T$. We refer to \cite[Chapter II]{G} for the definitions and explicit formulas, as well as for the fact that a general inner function decomposes as claimed above. 

\par There is a tiny---and particularly nice---class of inner functions that we need to single out, namely, the {\it finite Blaschke products}. These are rational functions of the form 
\begin{equation}\label{eqn:finblaprodegm}
z\mapsto c\prod_{j=1}^m\f{z-a_j}{1-\ov a_jz},
\end{equation}
where $a_1,\dots,a_m$ are points in $\D$ and $c$ is a unimodular constant. The number $m(\in\Z_+)$ is then the {\it degree} of the finite Blaschke product \eqref{eqn:finblaprodegm}. (In general, we define the degree of a rational function $R$ as the number of its poles---counting multiplicities---on the Riemann sphere $\C\cup\{\infty\}$, and we denote this number by $\deg R$.) When $m=0$, it is of course understood that \eqref{eqn:finblaprodegm} reduces to the constant function $c$. 

\par Our characterization of the extreme points of $\text{\rm ball}(H^1_{\mathcal K})$ splits into two conditions. First we verify that if a function $f\in H^1_{\mathcal K}$ is extreme for the ball, then the inner factor $I$ in its canonical factorization \eqref{eqn:fif} is a finite Blaschke product of degree not exceeding $M(=\#\mathcal K)$. In other words, we necessarily have 
$$I(z)=\prod_{j=1}^m\f{z-a_j}{1-\ov a_jz}$$
for some $a_1,\dots,a_m\in\D$, where $m(=\deg I)$ satisfies $0\le m\le M$. (The constant $c$ in \eqref{eqn:finblaprodegm} is now taken to be $1$; clearly, nothing is lost by doing so.)

\par Secondly, assuming that the inner factor $I$ of a unit-norm function $f\in H^1_{\mathcal K}$ has the above form, we find out what else is needed to make $f$ extreme. The answer is given in terms of a certain matrix $\mathfrak M$, built from $F$ (the outer factor of $f$) and the zeros $a_1,\dots,a_m$ of $I$ as described below. 

\par Consider the (outer) function 
\begin{equation}\label{eqn:deffzero}
F_0(z):=F(z)\prod_{j=1}^m(1-\ov a_jz)^{-2}
\end{equation}
and its coefficients 
\begin{equation}\label{eqn:defofck}
C_k:=\widehat F_0(k),\qquad k\in\Z. 
\end{equation}
Since $F_0\in H^1$, it follows in particular that $C_k=0$ for all $k<0$. We further define 
\begin{equation}\label{eqn:defakbk}
A(k):=\text{\rm Re}\,C_k,\qquad B(k):=\text{\rm Im}\,C_k\qquad(k\in\Z)
\end{equation}
and introduce, for $j=1,\dots,M$ and $l=0,\dots,m$, the numbers 
\begin{equation}\label{eqn:abplus}
A^+_{j,l}:=A(k_j+l-m)+A(k_j-l-m),\qquad B^+_{j,l}:=B(k_j+l-m)+B(k_j-l-m)
\end{equation}
and
\begin{equation}\label{eqn:abminus}
A^-_{j,l}:=A(k_j+l-m)-A(k_j-l-m),\qquad B^-_{j,l}:=B(k_j+l-m)-B(k_j-l-m).
\end{equation}
(The integers $k_j$ are, of course, those from \eqref{eqn:defmathcalk}.) Next, we build the $M\times(m+1)$ matrices 
\begin{equation}\label{eqn:plusmat}
\mathcal A^+:=\left\{A^+_{j,l}\right\},\qquad\mathcal B^+:=\left\{B^+_{j,l}\right\}
\end{equation}
and the $M\times m$ matrices
\begin{equation}\label{eqn:minusmat}
\mathcal A^-:=\left\{A^-_{j,l}\right\},\qquad
\mathcal B^-:=\left\{B^-_{j,l}\right\}.
\end{equation}
Here, the row index $j$ always runs from $1$ to $M$, whereas the column index $l$ runs from $0$ to $m$ for each of the two matrices in \eqref{eqn:plusmat}, and from $1$ to $m$ for each of those in \eqref{eqn:minusmat}. 
\par Finally, we construct the block matrix 
\begin{equation}\label{eqn:blockmatrix}
\mathfrak M=\mathfrak M_{\mathcal K}\left(F,\{a_j\}_{j=1}^m\right):=
\begin{pmatrix}
\mathcal A^+ & \mathcal B^- \\
\mathcal B^+ & -\mathcal A^-
\end{pmatrix},
\end{equation}
which has $2M$ rows and $2m+1$ columns. 

\par Our main result can now be stated readily.

\begin{thm}\label{thm:extrpoipuhone} Let $f\in H^1_{\mathcal K}$ be a function with $\|f\|_1=1$ whose canonical factorization is $f=IF$, with $I$ inner and $F$ outer. Then $f$ is an extreme point of $\text{\rm ball}(H^1_{\mathcal K})$ if and only if the following two conditions hold: 
\par{\rm (a)} $I$ is a finite Blaschke product, with $\deg I(=:m)$ not exceeding $M(=\#\mathcal K)$.
\par{\rm (b)} The matrix $\mathfrak M=\mathfrak M_{\mathcal K}\left(F,\{a_j\}_{j=1}^m\right)$, built as above from $F$ and the zeros $\{a_j\}_{j=1}^m$ of $I$, has rank $2m$.
\end{thm}

\par This criterion should be compared with its counterpart from \cite[Section 2]{DAdv2021}, where a similar rank condition on the appropriate matrix emerged in the context of lacunary polynomials. We also mention that Theorem \ref{thm:extrpoipuhone} was previously announced in \cite{DCRM}; a proof sketch was supplied there as well.

\par Of course, the original de Leeuw--Rudin theorem for the non-punctured $H^1$ space is recovered from Theorem \ref{thm:extrpoipuhone} by taking $\mathcal K=\emptyset$, in which case $M=0$ and condition (b) becomes void. As further examples, we now consider two special cases where Theorem \ref{thm:extrpoipuhone} is easy to apply.

\begin{exmp} Let $F\in H^1_{\mathcal K}$ be an outer function with $\|F\|_1=1$. Obviously enough, $F$ is then an extreme point of $\text{\rm ball}(H^1_{\mathcal K})$. This fact is simply a consequence of the de Leeuw--Rudin theorem, coupled with the inclusion $H^1_{\mathcal K}\subset H^1$, but we can also deduce it from Theorem \ref{thm:extrpoipuhone}. Indeed, applying the latter result to $f=F$, we have $I=1$ and $m=0$; therefore, $F_0=F(\in H^1_{\mathcal K})$ and 
$$C_{k_j}=\widehat F_0(k_j)=0\qquad(j=1,\dots,M).$$
It follows that the blocks $\mathcal A^+$ and $\mathcal B^+$ in \eqref{eqn:blockmatrix} reduce to zero columns, of height $M$ each, while the other two blocks are absent. Thus, $\text{\rm rank}\,\mathfrak M=0(=2m)$, so that conditions (a) and (b) are fulfilled. 
\end{exmp}

\begin{exmp} Suppose that the set $\mathcal K$ contains precisely one element, a positive integer which we call $k$ rather than $k_1$. Thus $M=1$, $\mathcal K=\{k\}$, and the space in question is 
$$H^1_{\{k\}}:=\{f\in H^1:\,\widehat f(k)=0\},$$
the punctured Hardy space with a single spectral hole. Now let $f=IF$ be a unit-norm function in $H^1_{\{k\}}$; as before, it is assumed that $I$ is inner and $F$ outer. If $I$ is constant, then $f$ is outer and hence extreme in $\text{\rm ball}\left(H^1_{\{k\}}\right)$. To characterize the \lq\lq interesting" (i.e., non-outer) extreme points of $\text{\rm ball}\left(H^1_{\{k\}}\right)$, we invoke Theorem \ref{thm:extrpoipuhone}. Condition (a) shows that we should only study the case where $m=1$, so we assume that 
\begin{equation}\label{eqn:ieqia}
I(z)=I_a(z):=\f{z-a}{1-\ov az}
\end{equation}
for some $a\in\D$. We then consider the function $F_0(z):=F(z)/(1-\ov az)^2$ and its Fourier coefficients 
$$\widehat F_0(n)=:C_n=A(n)+iB(n),\qquad n\in\Z.$$
More explicitly, 
\begin{equation}\label{eqn:convol}
C_n=\sum_{j=0}^n(j+1)\,\ov a^j\widehat F(n-j),
\end{equation}
with the understanding that the sum is zero for $n<0$. The matrix $\mathfrak M$ takes the form 
$$\mathfrak M=
\begin{pmatrix}
A_{1,0}^+ & A_{1,1}^+ & B_{1,1}^-\\
B_{1,0}^+ & B_{1,1}^+ & -A_{1,1}^-
\end{pmatrix},$$
where the entries are given by \eqref{eqn:abplus} and \eqref{eqn:abminus}, with $k_1=k$. Finally, the criterion for $f=I_aF$ to be an extreme point of $\text{\rm ball}\left(H^1_{\{k\}}\right)$ is that 
\begin{equation}\label{eqn:rkmeqtwo}
\text{\rm rank}\,\mathfrak M=2,
\end{equation}
as Theorem \ref{thm:extrpoipuhone} tells us. 
\par Now, the identity 
$$f(z)=(z-a)(1-\ov az)F_0(z)$$
allows us to rewrite the assumption $\widehat f(k)=0$ as 
$$aC_k-(1+|a|^2)C_{k-1}+\ov aC_{k-2}=0.$$
This in turn implies, upon separating the real and imaginary parts, that the first column of $\mathfrak M$ is a linear combination of the other two. Consequently, \eqref{eqn:rkmeqtwo} holds if and only if the determinant 
$$\Delta:=
\begin{vmatrix}
A_{1,1}^+ & B_{1,1}^-\\
B_{1,1}^+ & -A_{1,1}^-
\end{vmatrix}$$
is nonzero. A calculation reveals that $\Delta=|C_{k-2}|^2-|C_k|^2$. Thus,  \eqref{eqn:rkmeqtwo} boils down to saying that $|C_{k-2}|\ne|C_k|$, or equivalently, 
\begin{equation}\label{eqn:charonehole}
\left|\sum_{j=0}^{k-2}(j+1)\,\ov a^j\widehat F(k-2-j)\right|
\ne\left|\sum_{j=0}^k(j+1)\,\ov a^j\widehat F(k-j)\right|,
\end{equation}
this last restatement being due to \eqref{eqn:convol}. 
\par In summary, a unit-norm function in $H^1_{\{k\}}$ is an extreme point of $\text{\rm ball}\left(H^1_{\{k\}}\right)$ if and only if it is either outer or has the form $I_aF$, where $a\in\D$, the inner factor $I_a$ is given by \eqref{eqn:ieqia}, and $F\in H^1$ is an outer function satisfying \eqref{eqn:charonehole}.
\end{exmp}

\par Before stating our second theorem, we need to recall yet another geometric concept. Given a Banach space $X=(X,\|\cdot\|)$ and a point $x\in\text{\rm ball}(X)$, one says that $x$ is an {\it exposed point} of the ball 
if there exists a functional $\phi\in X^*$ of norm $1$ such that
$$\{y\in\text{\rm ball}(X):\,\phi(y)=1\}=\{x\}.$$
It is easy to show that every exposed point is extreme.

\par The next result provides a simple sufficient condition for a function in $H^1_{\mathcal K}$ to be an exposed point of the unit ball therein.

\begin{thm}\label{thm:expopoipuhone} If $f$ is an extreme point of $\text{\rm ball}(H^1_{\mathcal K})$ and if $1/f\in L^1$, then $f$ is an exposed point of $\text{\rm ball}(H^1_{\mathcal K})$. 
\end{thm}

\par There seems to be little hope for a complete---and reasonably explicit---description of the exposed points of $\text{\rm ball}(H^1_{\mathcal K})$, since even the classical case where $\mathcal K=\emptyset$ presents an open problem. In fact, the exposed points of $\text{\rm ball}(H^1)$ have been studied by a number of authors (see, e.g., \cite{Hel, Pol, Sar1, Sar2} and \cite[Section 3]{DJFA}, where various pieces of information were gathered), but no satisfactory characterization is currently available. Among the known facts we single out the following (see \cite{Sar1}): If $f$ is a unit-norm outer function in $H^1$ with $1/f\in L^1$, then $f$ is an exposed point of $\text{\rm ball}(H^1)$. It is this prototypical result that we now extend, by means of Theorem \ref{thm:expopoipuhone}, to the $H^1_{\mathcal K}$ setting. 

\par The plan for the rest of the paper is as follows. In Section 2 we collect some preliminary results, to be employed later. In Sections 3 and 4, we prove Theorem \ref{thm:extrpoipuhone}. This is done in two steps: first we establish the necessity of condition (a), and secondly, we show that among the functions satisfying (a), the extreme points are characterized by (b). Finally, Section 5 is devoted to proving Theorem \ref{thm:expopoipuhone}.

\section{Preliminaries}

Several lemmas will be needed. Before stating them, we list some of the function spaces that appear below and recall the appropriate definitions. 

\par Having already introduced the Hardy space $H^1$, we now define $H^p$ to be the intersection $H^1\cap L^p$ if $1<p\le\infty$, and the closure of $H^1$ in $L^p$ if $0<p<1$. The {\it Smirnov class} $N^+$ is the set of all ratios $\varphi/\psi$, where $\varphi$ ranges over $H^\infty$ and $\psi$ over the outer functions in $H^\infty$. (Equivalent---and more traditional---definitions of $H^p$ and $N^+$ can be found in \cite[Chapter II]{G}.) The functions in $H^1$ (resp., $L^1$) with finite spectrum will be referred to as polynomials (resp., trigonometric polynomials). Finally, we write $L^\infty_\R$ for the set of real-valued functions in $L^\infty$. 

\begin{lem}\label{lem:charextrhone} Let $X$ be a subspace of $H^1$. Suppose also that $f\in X$ is a function with $\|f\|_1=1$ whose canonical factorization is $f=IF$, with $I$ inner and $F$ outer. The following conditions are equivalent. 
\par{\rm (i.1)} $f$ is an extreme point of $\text{\rm ball}(X)$. 
\par{\rm (ii.1)} Whenever $h\in L^\infty_\R$ and $fh\in X$, we have $h=\const$ a.e. on $\T$.
\par{\rm (iii.1)} Whenever $G\in H^\infty$ is a function satisfying $G/I\in L^\infty_\R$ and $FG\in X$, we have $G=cI$ for some $c\in\R$. 
\end{lem}

\begin{proof} The equivalence between (i.1) and (ii.1) is well known (see, e.g., \cite [Chapter V, Section 9]{Gam} for the case of $X=H^1$). We nonetheless provide the details for the reader's convenience. 
\par Suppose that (ii.1) fails, meaning that there is a nonconstant function $h\in L^\infty_\R$ with $fh\in X$. We may assume in addition that $\int_\T|f|h=0$ and $\|h\|_\infty\le1$. (To achieve this, replace the original $h$ by $\al h+\be$ with suitable $\al,\be\in\R$ if necessary.) The functions $f_1:=f(1+h)$ and $f_2:=f(1-h)$ are then distinct unit-norm elements of $X$, and the identity $f=\f12(f_1+f_2)$ shows that $f$ is a non-extreme point of $\text{\rm ball}(X)$. This proves the implication (i.1)$\implies$(ii.1). 

\par Conversely, suppose that (i.1) fails. This tells us that there exists a non-null function $g\in X$ for which $\|f\pm g\|_1=1$. Putting $h:=g/f$, we have then 
\begin{equation}\label{eqn:babayaga}
\f1{2\pi}\int_\T|f(\ze)|\left\{|1+h(\ze)|+|1-h(\ze)|\right\}|d\ze|=2.
\end{equation}
Because $\|f\|_1=1$ and 
\begin{equation}\label{eqn:dedyaga}
|1+h(\ze)|+|1-h(\ze)|\ge2	
\end{equation}
wherever $h(\ze)$ is defined, \eqref{eqn:babayaga} is only possible if equality holds in \eqref{eqn:dedyaga} for almost all $\ze\in\T$. This in turn implies that $h$ takes values in the real interval $[-1,1]$; in particular, $h\in L^\infty_\R$. Also, $h$ must be nonconstant. (Otherwise, we would have $h\equiv c$ for some $c\in[-1,1]$ and the condition 
$$\|f+g\|_1=(1+c)\|f\|_1=1$$ 
would force $c$ to be $0$, meaning that $h\equiv0$ and hence $g\equiv0$.) Finally, since $fh=g\in X$, we see that (ii.1) is violated. The implication (ii.1)$\implies$(i.1) is thereby established. 

\par It remains to verify the equivalence of (ii.1) and (iii.1). Assuming that (ii.1) fails, we can find a nonconstant function $h\in L^\infty_\R$ for which the product $fh=:g$ is in $X$. Now put $G:=g/F$. Because $g$ and $F$ are both in $H^1$, while $F$ is outer, it follows that $G\in N^+$. Furthermore, we have
\begin{equation}\label{eqn:heqgovf}
h=\f gf=\f g{IF}=\f GI,
\end{equation}
whence $G=Ih\in L^\infty$; and since $N^+\cap L^\infty=H^\infty$ (see, e.g., \cite[Chapter II]{G}), we deduce that $G\in H^\infty$. 
Finally, in view of \eqref{eqn:heqgovf}, the assumptions 
\begin{equation}\label{eqn:condaaa}
h\in L^\infty_\R,\quad h\ne\const,\quad\text{\rm and}\quad g\in X
\end{equation}
take the form
\begin{equation}\label{eqn:condbbb}
G/I\in L^\infty_\R,\quad G/I\ne\const,\quad\text{\rm and}\quad FG\in X,
\end{equation}
meaning that condition (iii.1) fails. This proves that (iii.1) implies (ii.1). 
\par Conversely, if $G$ is an $H^\infty$ function making \eqref{eqn:condbbb} true, then \eqref{eqn:condaaa} holds with $h=G/I$ and $g=fh(=FG)$. Therefore, (ii.1) implies (iii.1). 
\end{proof}

\par\noindent{\it Remark.} The above lemma can be used to give a quick proof of the de Leeuw--Rudin theorem on the extreme points of $\text{\rm ball}(H^1)$. Indeed, for $X=H^1$, condition (iii.1) holds if and only if $I$ is constant (i.e., $f$ is outer). Here, the \lq\lq if" part is true because $H^\infty\cap L^\infty_\R$ contains only constants, while the converse is proved by taking $G=1+I^2$. 

\medskip Next, we establish an analogue of Lemma \ref{lem:charextrhone} for exposed points. 

\begin{lem}\label{lem:charexpohone} Under the assumptions of the preceding lemma, the following statements are equivalent. 
\par{\rm (i.2)} $f$ is an exposed point of $\text{\rm ball}(X)$. 
\par{\rm (ii.2)} Whenever $h$ is a nonnegative measurable function on $\T$ for which $fh\in X$, we have $h=\const$ a.e.
\par{\rm (iii.2)} Whenever $G\in N^+$ is a function satisfying $FG\in X$ and $G/I\ge0$ a.e. on $\T$, we have $G=cI$ for some constant $c\ge0$.
\end{lem}

\begin{proof} The equivalence between (i.2) and (ii.2) is a known fact (see, e.g., Lemma 1(B) in \cite{DMRL2000}). The underlying argument being short and simple, we reproduce it here for the sake of completeness. 
\par Suppose that $\phi\in X^*$ is a functional with $\|\phi\|=1$ and $\phi(f)=1$. An application of the Hahn--Banach theorem shows that $\phi$ has the form 
$$\phi(g)=\f1{2\pi}\int_\T u(\ze)g(\ze)\,|d\ze|,\qquad g\in X,$$
where $u=|f|/f$. It follows that, for a function $g\in\text{\rm ball}(X)$, the equality $\phi(g)=1$ occurs if and only if $\|g\|_1=1$ and 
\begin{equation}\label{eqn:argarg}
\f g{|g|}=\f f{|f|}\quad\text{\rm a.e. on }\T.
\end{equation}
Now, condition (i.2) can be expressed by saying that the only unit-norm function $g\in X$ satisfying \eqref{eqn:argarg} is $f$. This, in turn, is easily rephrased as (ii.2). 

\par To verify that (ii.2) is equivalent to (iii.2), we follow the pattern of the preceding proof; only minor adjustments are actually needed. 

\par Namely, if (ii.2) fails, then there is a nonconstant measurable function $h\ge0$ such that $fh=:g$ is in $X$. As before, we put $G:=g/F$. Since $g$ and $F$ are both in $H^1$, while $F$ is outer, we infer that $G\in N^+$. We also have identity \eqref{eqn:heqgovf} at our disposal. Consequently, the assumptions 
\begin{equation}\label{eqn:condccc}
h\ge0,\quad h\ne\const,\quad\text{\rm and}\quad g\in X
\end{equation}
take the form
\begin{equation}\label{eqn:condddd}
G/I\ge0,\quad G/I\ne\const,\quad\text{\rm and}\quad FG\in X,
\end{equation}
which means that condition (iii.2) fails. 
\par Conversely, if \eqref{eqn:condddd} holds for some $G\in N^+$, then \eqref{eqn:condccc} is fulfilled with $h=G/I$ and $g=fh(=FG)$. 
\end{proof}

\par Before proceeding with our next lemma, we need to introduce and discuss a concept that will be repeatedly used in what follows. 

\begin{defn} {\rm Given a nonnegative integer $N$ and a polynomial $p$, we say that $p$ is {\it $N$-symmetric} if}
\begin{equation}\label{eqn:symmnk}
\widehat p(N-k)=\ov{\widehat p(N+k)}
\end{equation}
{\rm for all }\,$k\in\Z$.	
\end{defn}

\par Equivalently, $p$ is {\it $N$-symmetric} if and only if the trigonometric polynomial $q:=z^{-N}p$ is real-valued on $\T$; indeed, \eqref{eqn:symmnk} tells us that $\widehat q(-k)=\ov{\widehat q(k)}$ for all $k\in\Z$. Also, it follows from \eqref{eqn:symmnk} that $\deg p\le2N$ and $\widehat p(N)\in\R$. Consequently, a polynomial $p$ is $N$-symmetric if and only if it is writable as 
\begin{equation}\label{eqn:symmpoly}
p(z)=\sum_{j=0}^{N-1}\ov\ga_{N-j}z^j+\sum_{j=N}^{2N}\ga_{j-N}z^j
\end{equation}
for some $\ga_0\in\R$ and $\ga_1,\dots,\ga_N\in\C$. Setting
$$\ga_0=2\al_0,\quad\ga_j=\al_j+i\be_j\quad(j=1,\dots,N),$$
where the $\al_j$'s and $\be_j$'s are real numbers, we may therefore identify the (generic) $N$-symmetric polynomial \eqref{eqn:symmpoly} with the vector 
$$(\al,\be):=(\al_0,\al_1,\dots,\al_N,\be_1,\dots,\be_N)$$
from $\R^{2N+1}$, to be called the {\it coefficient vector} of $p$. 

\begin{lem}\label{lem:genformreal} Given an integer $N\ge0$ and points $a_1,\dots,a_N\in\D$, let 
\begin{equation}\label{eqn:finblaprodaj}
B(z):=\prod_{j=1}^N\f{z-a_j}{1-\ov a_jz}.
\end{equation}
The general form of a function $\psi\in H^\infty$ with $\psi/B\in L^\infty_\R$ is then 
\begin{equation}\label{eqn:genformpsi}
\psi(z)=p(z)\prod_{j=1}^N(1-\ov a_jz)^{-2},
\end{equation}
where $p$ is an $N$-symmetric polynomial. 
\end{lem}

\par To keep on the safe side, we specify that the points $a_j$ above are not supposed to be pairwise distinct. Also, if $N=0$, then there are no $a_j$'s and the products in \eqref{eqn:finblaprodaj} and \eqref{eqn:genformpsi} are taken to be $1$, while $p$ reduces to a real constant. 

\par In the proof below, we use the notation 
\begin{equation}\label{eqn:modspa}
K_\th:=H^2\ominus\th H^2
\end{equation}
for the {\it star-invariant} (or {\it model}) subspace in $H^2$ generated by an inner function $\th$. It is well known (see \cite{DSS, N}) that \eqref{eqn:modspa}, with $\th$ inner, actually provides the general form of an invariant subspace for the backward shift operator 
$$f\mapsto\f{f-f(0)}z$$
in $H^2$. 
\par The following (fairly simple) fact can also be found in either \cite{DSS} or \cite{N}: If $\th$ is a finite Blaschke product, then $K_\th$ is formed by the rational functions $r$ whose poles (counted with multiplicities) are contained among those of $\th$ and which satisfy $\lim_{z\to\infty}r(z)/\th(z)=0$. In other words, if $\th$ is a finite Blaschke product of degree $n$, with zeros $\la_1,\dots,\la_n$ ($\in\D$), then $K_\th$ is the set of functions of the form 
$$z\mapsto p(z)\prod_{j=1}^n(1-\ov\la_jz)^{-1},$$
where $p$ is a polynomial with $\deg p\le n-1$. 

\medskip\noindent{\it Proof of Lemma \ref{lem:genformreal}.} If $\psi\in H^\infty$ with $\psi/B\in L^\infty_\R$, then 
$$\psi/B=\ov\psi/\ov B\quad\text{\rm a.e. on }\T,$$
or equivalently, $\psi\ov B^2=\ov\psi$. It follows that $\psi$ is orthogonal (in $H^2$) to the shift-invariant subspace $zB^2H^2$, and so $\psi\in K_\th$ with $\th=zB^2$. This $\th$ being a finite Blaschke product with $\deg\th=2N+1$, we know from the above discussion that $\psi$ is writable as
\begin{equation}\label{eqn:psipphi}
\psi(z)=p(z)\Phi(z),
\end{equation}
where 
\begin{equation}\label{eqn:defbigphi}
\Phi(z):=\prod_{j=1}^N(1-\ov a_jz)^{-2}
\end{equation}
and $p$ is a polynomial of degree at most $2N$. 

\par Further, we put 
\begin{equation}\label{eqn:defsss}
S(z):=\prod_{j=1}^N|z-a_j|^2
\end{equation}
and note that
\begin{equation}\label{eqn:bphisident}
\f{B(z)}{\Phi(z)}=z^NS(z),\qquad z\in\T,
\end{equation}
as verified by a straightforward calculation. Combining \eqref{eqn:psipphi} and \eqref{eqn:bphisident}, we see that
\begin{equation}\label{eqn:spsioverb}
S\psi/B=z^{-N}p\quad\text{\rm on }\T.
\end{equation}
Now, because the functions $S$ and $\psi/B$ are real-valued on $\T$, the same is true of the product $z^{-N}p$, and this means that $p$ is $N$-symmetric. The desired representation \eqref{eqn:genformpsi} is therefore provided by \eqref{eqn:psipphi}. 

\par Conversely, if $p$ is an $N$-symmetric polynomial and if $\psi$ is given by \eqref{eqn:psipphi}, then $\psi\in H^\infty$ and $\psi/B$ is real-valued (so that $\psi/B\in L^\infty_\R$) by virtue of \eqref{eqn:spsioverb}. The lemma is now proved. \qed

\begin{lem}\label{lem:posquo} Given an integer $N\ge0$ and points $a_1,\dots,a_N\in\D$, let $B$ be defined by \eqref{eqn:finblaprodaj}. Suppose also that $\psi\in H^{1/2}$ and $\psi/B\ge0$ a.e. on $\T$. Then $\psi$ has the form \eqref{eqn:genformpsi} for some $N$-symmetric polynomial $p$.
\end{lem}

\begin{proof} Once again, we define the functions $\Phi$ and $S$ by \eqref{eqn:defbigphi} and \eqref{eqn:defsss}. We also put $u:=\psi/\Phi$ and note that $u\in H^{1/2}$. The rest of the proof will consist in showing that $u$ is an $N$-symmetric polynomial. Once this is done, the desired representation \eqref{eqn:genformpsi} comes out readily; to arrive at it, we simply write $\psi=u\Phi$ and set $p:=u$. The lemma will thereby be established. 

\par We begin by recalling identity \eqref{eqn:bphisident}, which yields
\begin{equation}\label{eqn:znuznu}
z^{-N}u=S\Phi u/B=S\psi/B
\end{equation}
a.e. on $\T$. Since $S$ and $\psi/B$ are both nonnegative, the same is true of their product, and \eqref{eqn:znuznu} tells us that 
\begin{equation}\label{eqn:znupos}
z^{-N}u\ge0\quad\text{\rm a.e. on }\T.
\end{equation}
Now, a standard factoring technique (see \cite[Chapter II]{G}) allows us to write the function $u(\in H^{1/2})$ in the form $u=bv^2$, where $b$ is a Blaschke product and $v\in H^1$. In particular, since $|b|=1$, we have 
\begin{equation}\label{eqn:moduone}
|u|=|v|^2=v\ov v
\end{equation}
(here and below, everything is assumed to hold a.e. on $\T$). On the other hand, \eqref{eqn:znupos} gives 
\begin{equation}\label{eqn:modutwo}
|u|=uz^{-N}=bv^2z^{-N}.
\end{equation}
A juxtaposition of \eqref{eqn:moduone} and \eqref{eqn:modutwo} reveals that $v\ov v=bv^2z^{-N}$, or equivalently, 
\begin{equation}\label{eqn:antianalytic}
\ov v=bvz^{-N}.
\end{equation}
Because the functions $v$ and $bv$ are both in $H^1$, their spectra are contained in $[0,\infty)$. It follows that 
\begin{equation}\label{eqn:twospectravbv}
\text{\rm spec}\,\ov v\subset(-\infty,0]\quad\text{\rm and}\quad
\text{\rm spec}\left(bvz^{-N}\right)\subset[-N,\infty).
\end{equation}
At the same time, \eqref{eqn:antianalytic} shows that the two spectra in \eqref{eqn:twospectravbv} are actually equal, so they are both contained in $[-N,0]$. This in turn implies that 
$$\text{\rm spec}\,v\subset[0,N]\quad\text{\rm and}\quad
\text{\rm spec}(bv)\subset[0,N].$$
In other words, $v$ and $bv$ are polynomials, of degree at most $N$ each. 
Consequently, their product (which is $u$) is a polynomial of degree at most $2N$. Moreover, $u$ is an {\it $N$-symmetric} polynomial, because $z^{-N}u$ is real-valued by virtue of \eqref{eqn:znupos}. The proof is now complete.
\end{proof}

\section{Proof of Theorem \ref{thm:extrpoipuhone}: Step 1}

This step consists in proving the necessity of condition (a) in Theorem \ref{thm:extrpoipuhone}. Thus, we want to show that a unit-norm function $f(=IF)\in H^1_{\mathcal K}$ will be a non-extreme point of the unit ball whenever it violates (a). 

\par Assume that condition (a) fails, so that $I$ does not reduce to a finite Blaschke product of degree at most $M$. This means that $I$ is divisible either by a (finite or infinite) Blaschke product with at least $M+1$ zeros, or by a nontrivial singular inner function. In either case, Frostman's theorem (see \cite[Chapter II]{G}) tells us that there exists a point $w\in\D$ for which 
\begin{equation}\label{eqn:frostshift}
\ph:=\f{I-w}{1-\ov wI}
\end{equation}
is a Blaschke product. Moreover, our current assumption on $I$ ensures that $\ph$ has at least $M+1$ zeros. (Otherwise, if $\ph$ were a finite Blaschke product of degree $d$, with $d\le M$, then so would be $I$. Indeed, the identity $I=(\ph+w)/(1+\ov w\ph)$ would then imply that $I$ is analytic on $\T$, so $I$ would have to be a finite Blaschke product. Since $|\ph(\ze)|=1>|w|$ for $\ze\in\T$, an application of Rouch\'e's theorem would furthermore show that $\ph+w$ has precisely $d$ zeros in $\D$, and the same would be true for $I$.)
\par Consequently, we can find a factorization 
\begin{equation}\label{eqn:phiphiphi}
\ph=\ph_1\ph_2, 
\end{equation}
where both factors on the right are Blaschke products (hence subproducts of $\ph$) and $\ph_1$ has precisely $M+1$ zeros. Setting $N:=M+1$, we may thus write 
\begin{equation}\label{eqn:finblaprodphione}
\ph_1(z)=\prod_{j=1}^N\f{z-a_j}{1-\ov a_jz}
\end{equation}
with the appropriate $a_1,\dots,a_N\in\D$. Next, we define the function $g\in H^\infty$ by the formula $g:=1-\ov wI$ and infer from \eqref{eqn:frostshift} that 
\begin{equation}\label{eqn:quoquo}
I/\ph=g/\ov g
\end{equation}
a.e. on $\T$. Finally, we combine \eqref{eqn:phiphiphi} and \eqref{eqn:quoquo} to get 
\begin{equation}\label{eqn:ieqprod}
I=\ph_1\ph_2g/\ov g.
\end{equation}

\par Our plan is to prove that $f$ is a non-extreme point of $\text{\rm ball}(H^1_{\mathcal K})$ by verifying that it violates condition (iii.1) of Lemma \ref{lem:charextrhone}, with $X=H^1_{\mathcal K}$. Thus, we need to produce a function $G\in H^\infty$, other than a constant multiple of $I$, with the properties that 
\begin{equation}\label{eqn:firpro}
G/I\in L^\infty_\R
\end{equation}
and 
\begin{equation}\label{eqn:secpro}
FG\in H^1_{\mathcal K}.
\end{equation}
It turns out that such a $G$ can be constructed in the form 
\begin{equation}\label{eqn:intheform}
G=g^2p\Phi\ph_2,
\end{equation}
where 
$$\Phi(z):=\prod_{j=1}^N(1-\ov a_jz)^{-2}$$
and $p$ is an $N$-symmetric polynomial; this claim will be justified below.

\par First of all, \eqref{eqn:intheform} actually defines an $H^\infty$ function, since each of the factors on the right-hand side is in $H^\infty$. Furthermore, any such $G$ will satisfy \eqref{eqn:firpro}. Indeed, we may combine \eqref{eqn:intheform} and \eqref{eqn:ieqprod} to find that 
\begin{equation}\label{eqn:capgovercapi}
G/I=G\ov I=g^2p\Phi\ph_2\cdot\ov\ph_1\ov\ph_2\ov g/g=|g|^2p\Phi\ov\ph_1
\end{equation}
a.e. on $\T$. An application of Lemma \ref{lem:genformreal} with $B=\ph_1$ now yields
$$p\Phi\ov\ph_1\left(=p\Phi/\ph_1\right)\in L^\infty_\R.$$
The product $|g|^2p\Phi\ov\ph_1$ is therefore also in $L^\infty_\R$, and \eqref{eqn:firpro} is readily implied by \eqref{eqn:capgovercapi}. 

\par So far, everything was valid for an arbitrary $N$-symmetric polynomial $p$. Now, we shall see that the appropriate choice of $p$ in \eqref{eqn:intheform} will ensure \eqref{eqn:secpro}, along with the condition
\begin{equation}\label{eqn:gineconst}
G/I\ne\const.
\end{equation}
Multiplying \eqref{eqn:intheform} by $F$ gives 
\begin{equation}\label{eqn:fgfzerop}
FG=F_0p,
\end{equation}
where 
$$F_0:=Fg^2\Phi\ph_2\,(\in H^1).$$
For \eqref{eqn:secpro} to hold, it is necessary and sufficient that 
$$\widehat{(FG)}(k_j)=0\quad\text{\rm for }j=1,\dots,M.$$
Equivalently, in view of \eqref{eqn:fgfzerop}, the numbers 
\begin{equation}\label{eqn:cjfourier}
\de_j:=\widehat{(F_0p)}(k_j),\qquad j=1,\dots,M,
\end{equation}
must be null. 

\par On the other hand, we know from the previous section that there is a natural isomorphism between the space of $N$-symmetric polynomials and $\R^{2N+1}$. Namely, the general form of an $N$-symmetric polynomial $p$ is given by 
$$p(z)=p_{(\al,\be)}(z):=
\sum_{l=0}^{N-1}\left(\al_{N-l}-i\be_{N-l}\right)z^l+2\al_0z^N+
\sum_{l=N+1}^{2N}\left(\al_{l-N}+i\be_{l-N}\right)z^l,$$
where 
\begin{equation}\label{eqn:albevect}
(\al,\be):=\left(\al_0,\al_1,\dots,\al_N,\be_1,\dots,\be_N\right)\in\R^{2N+1}.
\end{equation}
With this in mind, we begin by taking an arbitrary vector \eqref{eqn:albevect} and then define, for $1\le j\le M$, the numbers $\de_j(\al,\be)$ as in \eqref{eqn:cjfourier}, but with $p=p_{(\al,\be)}$. That is, 
$$\de_j(\al,\be):=\widehat{(F_0p_{(\al,\be)})}(k_j),\qquad j=1,\dots,M.$$
Finally, we consider the $\R$-linear operator $T:\R^{2N+1}\to\R^{2M}$ that acts by the rule 
$$T(\al,\be)=\left(\text{\rm Re}\,\de_1(\al,\be),\,\text{\rm Im}\,\de_1(\al,\be),\dots,
\text{\rm Re}\,\de_M(\al,\be),\,\text{\rm Im}\,\de_M(\al,\be)\right).$$

\par It is the dimension of the subspace $\mathfrak N_T:=\text{\rm ker}\,T$, the kernel of $T$ in $\R^{2N+1}$, that interests us here. The rank-nullity theorem (see, e.g., \cite[p.\,63]{Axl}) yields
$$\text{\rm rank}\,T+\text{\rm dim}\,\mathfrak N_T=2N+1,$$
and we combine this with the obvious inequality
$$\text{\rm rank}\,T\le2M=2N-2$$
to deduce that 
$$\text{\rm dim}\,\mathfrak N_T\ge3.$$
In particular, we can find two linearly independent vectors, say $(\al^{(1)},\be^{(1)})$ and $(\al^{(2)},\be^{(2)})$, in $\mathfrak N_T$. The corresponding $N$-symmetric polynomials, which we denote for simplicity by $p_1$ and $p_2$, are then linearly independent as well. Consequently, at least one of them (let it be $p_1$) is not a constant multiple of $I/(g^2\Phi\ph_2)$, whence
$$g^2p_1\Phi\ph_2/I\ne\const.$$
Also, because the coefficient vector $(\al^{(1)},\be^{(1)})$ of $p_1$ is in $\mathfrak N_T$, the numbers \eqref{eqn:cjfourier} are null for $p=p_1$. This choice of $p$ therefore guarantees that the product on either side of \eqref{eqn:fgfzerop} belongs to $H^1_{\mathcal K}$. 

\par In summary, setting $p=p_1$ in \eqref{eqn:intheform} we arrive at a function $G\in H^\infty$ that satisfies \eqref{eqn:firpro}, \eqref{eqn:secpro} and \eqref{eqn:gineconst}. We conclude that condition (iii.1) of Lemma \ref{lem:charextrhone} breaks down for $X=H^1_{\mathcal K}$, and so $f$ fails to be an extreme point of $\text{\rm ball}(H^1_{\mathcal K})$.

\section{Proof of Theorem \ref{thm:extrpoipuhone}: Step 2}

This second step consists in characterizing the extreme points of $\text{\rm ball}(H^1_{\mathcal K})$ among those unit-norm functions $f(=IF)$ in $H^1_{\mathcal K}$ which obey condition (a). Thus, the inner factor $I$ of $f$ is now assumed to be of the form 
$$I(z)=\prod_{j=1}^m\f{z-a_j}{1-\ov a_jz},$$
where $0\le m\le M$ and $a_1,\dots,a_m\in\D$. 
\par In view of the equivalence relation (i.1)$\iff$(iii.1) from Lemma \ref{lem:charextrhone}, our purpose is to find out whether there exists a function $G\in H^\infty$ (other than a constant multiple of $I$) such that 
\begin{equation}\label{eqn:govireal}
G/I\in L^\infty_\R
\end{equation}
and 
\begin{equation}\label{eqn:prodfghonek}
FG\in H^1_{\mathcal K}.
\end{equation}
From Lemma \ref{lem:genformreal} we know that the functions $G\in H^\infty$ satisfying \eqref{eqn:govireal} are precisely those of the form 
\begin{equation}\label{eqn:geqpphi}
G=p\Phi_0,
\end{equation}
where $p$ is an $m$-symmetric polynomial and 
$$\Phi_0(z):=\prod_{j=1}^m(1-\ov a_jz)^{-2}.$$
We further need to determine which choices of $p$ ensure \eqref{eqn:prodfghonek}. Assuming \eqref{eqn:geqpphi}, we put $F_0:=F\Phi_0(\in H^1)$ and rewrite condition \eqref{eqn:prodfghonek} as $pF_0\in H^1_{\mathcal K}$, which in turn boils down to 
\begin{equation}\label{eqn:bober}
\widehat{(pF_0)}(k_j)=0\quad\text{\rm for}\quad j=1,\dots,M.
\end{equation}
(It should be noted that our current $F_0$ agrees with its namesake from Section 1.)
\par Next, we want to recast equations \eqref{eqn:bober} in terms of the coefficient vector of $p$. To this end, we first write $p$ in the form \eqref{eqn:symmpoly} (with $m$ in place of $N$), which gives 
\begin{equation}\label{eqn:symmpolybis}
p(z)=\sum_{l=0}^{m-1}\ov\ga_{m-l}z^l+\sum_{l=m}^{2m}\ga_{l-m}z^l
\end{equation}
for some $\ga_0\in\R$ and $\ga_1,\dots,\ga_m\in\C$. Using the notation 
$$C_r:=\widehat F_0(r),\qquad r\in\Z$$
(in accordance with \eqref{eqn:defofck}), we find then, for any fixed $k\in\Z$, that
\begin{equation}\label{eqn:pfhatk}
\begin{aligned}
\widehat{(pF_0)}(k)&=\sum_{l=0}^{2m}\widehat F_0(k-l)\widehat p(l)
=\sum_{l=0}^{m-1}C_{k-l}\,\ov\ga_{m-l}
+\sum_{l=m}^{2m}C_{k-l}\,\ga_{l-m}\\
&=\sum_{l=1}^{m}C_{k+l-m}\,\ov\ga_l+\sum_{l=0}^{m}C_{k-l-m}\,\ga_l.
\end{aligned}
\end{equation}
Therefore, equations \eqref{eqn:bober} take the form
\begin{equation}\label{eqn:dvabobra}
\sum_{l=0}^{m}C_{k_j-l-m}\,\ga_l+\sum_{l=1}^{m}C_{k_j+l-m}\,\ov\ga_l=0
\qquad(j=1,\dots,M).
\end{equation}
We now write
\begin{equation}\label{eqn:creqarbr}
C_r=A(r)+iB(r)\quad\text{\rm for}\quad r\in\Z
\end{equation}
(in accordance with \eqref{eqn:defakbk}) and decompose the $\ga_l$'s similarly. Precisely speaking, we put
\begin{equation}\label{eqn:galbebis}
\ga_0=2\al_0,\quad\ga_l=\al_l+i\be_l\quad\text{\rm for}\quad l=1,\dots,m,
\end{equation}
where the $\al_l$'s and $\be_l$'s are real. Finally, we plug \eqref{eqn:creqarbr} and \eqref{eqn:galbebis} into \eqref{eqn:dvabobra} to obtain, after separating the real and imaginary parts, a system of $2M$ real equations. Namely, these are 
\begin{equation}\label{eqn:reparteqbis}
\sum_{l=0}^mA^+_{j,l}\,\al_l+\sum_{l=1}^mB^-_{j,l}\,\be_l=0\qquad(j=1,\dots,M)
\end{equation}
and 
\begin{equation}\label{eqn:imparteqbis}
\sum_{l=0}^mB^+_{j,l}\,\al_l-\sum_{l=1}^mA^-_{j,l}\,\be_l=0\qquad(j=1,\dots,M),
\end{equation}
where the notations \eqref{eqn:abplus} and \eqref{eqn:abminus} are being used. 

\par These equations tell us that the vector 
\begin{equation}\label{eqn:vectoralbebis}
(\al,\be):=(\al_0,\al_1,\dots,\al_m,\be_1,\dots,\be_m)
\end{equation}
(i.e., the coefficient vector of $p$) belongs to the subspace 
\begin{equation}\label{eqn:nullspacebis}
\mathcal N:=\text{\rm ker}\,\mathfrak M,
\end{equation}
the kernel of the linear map $\mathfrak M:\R^{2m+1}\to\R^{2M}$ defined by \eqref{eqn:blockmatrix}. 

\par To summarize, the functions $G\in H^\infty$ satisfying \eqref{eqn:govireal} and \eqref{eqn:prodfghonek} are precisely those of the form \eqref{eqn:geqpphi}, 
where $p=p_{(\al,\be)}$ is an $m$-symmetric polynomial whose coefficient vector \eqref{eqn:vectoralbebis} is in $\mathcal N$. (We write $p_{(\al,\be)}$ for the polynomial \eqref{eqn:symmpolybis} with coefficients $\ga_0,\dots,\ga_m$ given by \eqref{eqn:galbebis}.) The functions $G$ of interest are thereby nicely parametrized by vectors from $\mathcal N$, and it is the dimension of $\mathcal N$ that we should now look at. 

\par First of all, we always have $\dim\mathcal N\ge1$. Indeed, setting $G=I$ obviously makes \eqref{eqn:govireal} and \eqref{eqn:prodfghonek} true. The corresponding $m$-symmetric polynomial in \eqref{eqn:geqpphi} is then 
$$\wt p(z):=I(z)/\Phi_0(z)=\prod_{j=1}^m(z-a_j)(1-\ov a_jz),$$
so its coefficient vector, say $(\wt\al,\wt\be)$, is a non-null element of $\mathcal N$. Now, if $\dim\mathcal N=1$, then $\mathcal N$ is spanned by $(\wt\al,\wt\be)$, and the only possible polynomials $p$ in \eqref{eqn:geqpphi} are constant multiples of $\wt p$; equivalently, the only functions $G\in H^\infty$ that obey \eqref{eqn:govireal} and \eqref{eqn:prodfghonek} are constant multiples of $I$. On the other hand, if $\dim\mathcal N>1$, then we can find a vector $(\al,\be)\in\mathcal N$ which is not a scalar multiple of $(\wt\al,\wt\be)$; plugging the corresponding $m$-symmetric polynomial $p=p_{(\al,\be)}$ into \eqref{eqn:geqpphi}, we arrive at a function $G\in H^\infty$ with properties \eqref{eqn:govireal} and \eqref{eqn:prodfghonek} for which $G/I\ne\const$. 

\par By virtue of Lemma \ref{lem:charextrhone}, we can now conclude that a unit-norm function $f=IF\in H^1_{\mathcal K}$ satisfying condition (a) is an 
extreme point of $\text{\rm ball}(H^1_{\mathcal K})$ if and only if the kernel $\mathcal N(\subset\R^{2m+1})$ of the associated linear map $\mathfrak M$ has dimension $1$. Finally, we know from the rank-nullity theorem (see, e.g., \cite[p.\,63]{Axl}) that 
$$\text{\rm rank}\,\mathfrak M+\text{\rm dim}\,\mathcal N=2m+1,$$
so we may restate the condition $\dim\mathcal N=1$ as $\text{\rm rank}\,\mathfrak M=2m$. This completes the proof.

\section{Proof of Theorem \ref{thm:expopoipuhone}}

Let $f$ be a function satisfying the theorem's hypotheses. As before, we write $f=IF$ with $I$ inner and $F$ outer. Because $f$ is an extreme point of $\text{\rm ball}(H^1_{\mathcal K})$, we know from Theorem \ref{thm:extrpoipuhone} that $I$ is a finite Blaschke product with $\deg I(=:m)$ not exceeding $M$. Thus, 
$$I(z)=\prod_{j=1}^m\f{z-a_j}{1-\ov a_jz}$$
for certain $a_1,\dots,a_m\in\D$. 

\par Our plan is to prove that $f$ is an exposed point of $\text{\rm ball}(H^1_{\mathcal K})$ by verifying condition (iii.2) of Lemma \ref{lem:charexpohone}, with $X=H^1_{\mathcal K}$. To this end, assume that $G\in N^+$ is a function for which 
\begin{equation}\label{eqn:firproexpo}
G/I\ge0\quad\text{\rm a.e. on }\T
\end{equation}
and 
\begin{equation}\label{eqn:secproexpo}
FG\in H^1_{\mathcal K}.
\end{equation}
Clearly, the function $U:=FG$ is then {\it a fortiori} in $L^1$; we also have $1/F\in L^1$ (since $1/f\in L^1$ by hypothesis, while $|F|=|f|$ a.e. on $\T$), and we combine the two facts to infer that $G=U/F\in L^{1/2}$. This in turn implies that $G$ is actually in $H^{1/2}$ ($=N^+\cap L^{1/2}$). 

\par We may now apply Lemma \ref{lem:posquo}, with $G$ and $I$ in place of $\psi$ and $B$, to conclude that 
$$G(z)=p(z)\prod_{j=1}^m(1-\ov a_jz)^{-2}$$
for some $m$-symmetric polynomial $p$. As a consequence, we see that $G\in H^\infty$. On the other hand, being an extreme point of $\text{\rm ball}(H^1_{\mathcal K})$, the function $f$ obeys condition (iii.1) of Lemma \ref{lem:charextrhone} with $X=H^1_{\mathcal K}$. This means that {\it every} function $G\in H^\infty$ satisfying \eqref{eqn:secproexpo} and making $G/I$ real-valued a.e. on $\T$ is given by $G=cI$ for some $c\in\R$. In particular, our current $G$ is necessarily of this form, the constant $c$ being actually nonnegative in view of \eqref{eqn:firproexpo}. 

\par We have thereby checked condition (iii.2) of Lemma \ref{lem:charexpohone}, with $X=H^1_{\mathcal K}$. The lemma then tells us that $f$ is an exposed point of $\text{\rm ball}(H^1_{\mathcal K})$, and the proof is complete.


\medskip
 
\begin{thebibliography}{12}

\bibitem{Axl} S. Axler, {\it Linear algebra done right}, Third edition, Springer, Cham, 2015.

\bibitem{dLR} K. de Leeuw and W. Rudin, {\it Extreme points and extremum problems in $H_1$}, Pacific J. Math. \textbf{8} (1958), 467--485.

\bibitem{DSS} R. G. Douglas, H. S. Shapiro, and A. L. Shields, {\it Cyclic vectors and invariant subspaces for the backward shift operator}, Ann. Inst. Fourier (Grenoble) \textbf{20} (1970), 37--76. 

\bibitem{DKal} K. M. Dyakonov, {\it The geometry of the unit ball in the space $K^1_\theta$}, Geometric problems of the theory of functions and sets, 52--54, Kalinin. Gos. Univ., Kalinin, 1987. (Russian)

\bibitem{DSib} K. M. Dyakonov, {\it Moduli and arguments of analytic functions from subspaces in $H^p$ that are invariant under the backward shift operator}, Sibirsk. Mat. Zh. \textbf{31} (1990), no. 6, 64--79; translation in Siberian Math. J. \textbf{31} (1990), 926--939.

\bibitem{DPAMS} K. M. Dyakonov, {\it Interpolating functions of minimal norm, star-invariant subspaces, and kernels of Toeplitz operators}, Proc. Amer. Math. Soc. \textbf{116} (1992), 1007--1013. 

\bibitem{DJFA} K. M. Dyakonov, {\it Kernels of Toeplitz operators via Bourgain's factorization theorem}, J. Funct. Anal. \textbf{170} (2000), 93--106.

\bibitem{DMRL2000} K. M. Dyakonov, {\it Polynomials and entire functions: zeros and geometry of the unit ball}, Math. Res. Lett. \textbf{7} (2000), 393--404.

\bibitem{DAMP} K. M. Dyakonov, {\it An extremal problem for functions annihilated by a Toeplitz operator}, Anal. Math. Phys. \textbf{9} (2019), 1019--1029.

\bibitem{DAdv2021} K. M. Dyakonov, {\it Lacunary polynomials in $L^1$: geometry of the unit sphere}, Adv. Math. \textbf{381} (2021), 107607, 24 pp.

\bibitem{DCRM} K. M. Dyakonov, {\it A Rudin--de Leeuw type theorem for functions with spectral gaps}, C. R. Math. Acad. Sci. Paris \textbf{359} (2021), 797--803.

\bibitem{Gam} T. W. Gamelin, {\it Uniform algebras}, Prentice-Hall, Englewood Cliffs, NJ, 1969. 

\bibitem{G} J. B. Garnett, {\it Bounded analytic functions}, Revised first edition, Springer, New York, 2007.

\bibitem{Hel} H. Helson, {\it Large analytic functions. II}, Analysis and partial differential equations, 217--220, Lecture Notes in Pure and Appl. Math., 122, Dekker, New York, 1990.

\bibitem{Hof} K. Hoffman, {\it Banach spaces of analytic functions}, Prentice-Hall, Englewood Cliffs, NJ, 1962.

\bibitem{N} N. K. Nikolski, {\it Operators, Functions, and Systems: An Easy Reading, Vol. 2: Model operators and systems}, Mathematical Surveys and Monographs, 93, Amer. Math. Soc., Providence, RI, 2002. 

\bibitem{Pol} A. G. Poltoratski, {\it Properties of exposed points in the unit ball of $H^1$}, Indiana Univ. Math. J. \textbf{50} (2001), 1789--1806.

\bibitem{Sar1} D. Sarason, {\it Exposed points in $H^1$. I}, The Gohberg anniversary collection, Vol. II (Calgary, AB, 1988), 485--496, Oper. Theory Adv. Appl., 41, Birkh\"auser, Basel, 1989.

\bibitem{Sar2} D. Sarason, {\it Exposed points in $H^1$. II}, Topics in operator theory: Ernst D. Hellinger memorial volume, 333--347, Oper. Theory Adv. Appl., 48, Birkh\"auser, Basel, 1990.


\end{thebibliography}
\end{document}